\documentclass[11pt]{article}

\usepackage[top=1in, bottom=1in, left=1in, right=1in]{geometry}
\usepackage{amsmath,amsfonts}
\usepackage{graphicx}
\usepackage[bf]{caption}

\usepackage[utf8]{inputenc}

\usepackage{geometry}
\usepackage{graphicx}
\usepackage{float}
\usepackage{titling}

\newtheorem{theorem}{Theorem}
\newtheorem{corollary}[theorem]{Corollary}
\newtheorem{remark}[theorem]{Remark}
\newenvironment{proof}{{\sc Proof:}}{
~\hfill\rule{2mm}{3mm}\vspace{.1in}}

\thanksmarkseries{arabic}

\begin{document}

\title{Hopf Bifurcations in Fast/Slow Systems with\\ Rate-Dependent Tipping }

\author{Jonathan Hahn \thanks{School of Mathematics, University of Minnesota, Minneapolis, MN 55455, USA (hahnx240@umn.edu) \newline April 17, 2017}}

%\date{October 28, 2016}
\date{}

\maketitle
\begin{abstract}
We analyze rate-dependent tipping in a fast/slow system with an equilibrium near the fold of a critical manifiold. We find a Hopf bifurcation as the rate parameter increases in the reduced co-moving system. This implies the growth of a limit cycle as the system changes from tracking a quasi-static equilibrium to tipping. Rather than trajectories diverging at a critical rate, they continue to track the quasi-static equilibrium in a spiral corresponding to an emerging limit cycle at the Hopf bifurcation. We apply the same analysis to a forced van der Pol oscillator to show this phenomenon in a familiar system where the growth of this limit cycle is well understood. 
\end{abstract}

\section{Introduction}
Simple examples have often motivated the understanding of phenomena that may arise in more complex or general settings. In this paper, we analyze a simple fast/slow system that appeared in a paper by Ashwin et. al. \cite{Ashwin2012} which investigated representative examples for a general theory of rate-dependent tipping. Ashwin et. al. used geometric singular perturbation (GSP) theory to determine a critical rate of tipping as a small parameter, $\epsilon$, approached $0$ in the fast/slow system. GSP is a powerful and often-used technique for analyzing a fast/slow system; however, we wish to demonstrate that it may not give us a full understanding of the system for $\epsilon > 0$ in this case.

We seek to provide a more complete picture of the rate-induced tipping in this system. Instead of using GSP-style analysis, we reduce to a two-dimensional autonomous system and use simple bifurcation analysis to examine it for $\epsilon > 0$. The system in question is not complicated and does not require any ground-breaking techniques to analyze. Yet, it is important to analyze such a simple system because it demonstrates a result which may generalize to more complicated fast/slow systems.

The example concerns a steadily moving {\it quasi-static equilibrium} (QSE) near the fold of a folded critical manifold in a 2-dimensional fast/slow system. The GSP analysis of Ashwin et. al. provides a critical rate of tipping for $\epsilon \to 0$, which gives a good approximation for determining when trajectories slip over the fold, tipping away from the equilibrium. However, this analysis gives the impression that there exists a bifurcation in this system where, at the critical rate, a globally stable equilibrium becomes an repelling equilibrium and all trajectories head toward infinity, or at least, far away from the QSE. As the system is non-autonomous, to our knowledge there is no existing theory to dispute this possibility. However, we can reduce the system into an autonomous one, where the rate parameter is a standard bifurcation parameter. Index theory of attractors and repellers rules out the possibility for a single global attractor to disappear in this way.

We will show that the reduced system actually undergoes a Hopf bifurcation at the critical rate, and an asymptotically stable limit cycle is born from the destabilizing equilibrium. This limit cycle grows continuously, meaning that trajectories do not simply bifurcate from tracking an equilibrium to approaching infinity. The distance of the limit cycle from the equilibrium grows continuously with the rate parameter. The limit cycle grows quickly, which gives the impression of a tipping event, but it is still continuous rather than immediate. In the original system, the bifurcation results in trajectories which spiral around the quasi-static equilibrium. As the rate parameter increases further, the limit cycle in the reduced system increases in size, and in the original system the orbits will spiral with increasing amplitude. 

The behavior of this system is actually an instance of a canard phenomenon, whose prototypical example is the van der Pol oscillator. In early days it was thought that the van der Pol system changed from having a stable equilibrium to a large limit cycle immediately upon bifurcation because that is what happened in numerical simulations. It is now known that the system actually goes through a transition where canard cycles increase to the large limit cycle during a very small range of the bifurcation parameter. The same sort of canard phenomenon occurs in the fast/slow system that Ashwin et. al. analyze.

In this paper we mainly seek to demonstrate that in the case of rate induced tipping with the with steady drift of an equilibrium near a folded critical manifold, there may be a continuous transition from trajectories that stay on the stable side of the fold to trajectories that pass over the fold and spiral around it. This continuous transition makes it more ambiguous or arbitrary to define a tipping threshold, since there is no immediate jump in the distance of trajectories from the moving equilibrium. The steady drift of the system allows for easy analysis, but we believe the same behavior will arise even in fast/slow sytems that we cannot simplify such ease. We also present this as an example displaying the need of an index theory for attractors and repellers in non-autonomous systems.

\section{Rate-Induced Bifurcations}
Bifurcations are conventionally viewed as changes in the stability of an equilibrium or limit cycle as a parameter varies. Rate-dependent or rate-induced bifurcations are distinct in that the stability of a system does not change with the variation of a parameter; rather, a parameter varies too quickly for the state to follow the movement of an equilibrium point. In many examples of rate-dependent tipping, there exists a critical rate defining a threshold for the rate of the parameter. Below the critical rate, the state is able track the equilibrium, and past the critical rate the system tips and fails to track it. 

Consider a system with state vector $x \in \mathbb{R}^n$, with parameters $\mu \in \mathbb{R}^k$ that do not vary over time, and a time-dependent external forcing function $\lambda(rt) \in \mathbb{R}^l$:
\[\frac{dx}{dt} = f(x, \mu, \lambda(rt))\]
Here $r$ is the rate of the forcing function. We will assume that for every fixed value of $\lambda$ in some domain, the system has a stable equilibrium, $\tilde{x}$. We will also assume that the stable state depends continuously on $\lambda$. So the equilibrium of the system can be written as $\tilde{x}(\lambda)$. For small values of $r$, the state will track the moving equilibrium, $\tilde{x}(\lambda(rt))$, called the {\em quasi-static equilibrium} (QSE). In some cases, for $r$ large enough, the quasi-static equilbrium moves too quickly for the system to follow. In this situation, we say there is a rate-induced bifurcation. 

Authors like Ashwin et. al. and Perryman \cite{Ashwin2012, Ashwin2015, Perryman2015} use various criteria to determine when a rate-induced tipping point has occured. Some cases they examine with a {\em tipping radius} $R$, in which the system tips if the state reaches a distance of $R$ from the QSE. In other cases, there may be a topological change in the system which causes the state to move away from the QSE, or to the basin of attraction of another equilibrium, as in \cite{Scheffer2008}. Sometimes, as in \cite{Wieczorek2011}, tipping occurs when an excursion happens due to the formation of a canard trajectory. The criterion for tipping can be arbitrary in some cases, then, and can depend on the physical meaning of the system rather than topological properties of the system.

We will look at two fast-slow systems which have a rate-induced bifurcation phenomenon. First we examine the system that appeared in \cite{Ashwin2012}. Next we analyze a system similar to the forced van der Pol oscillator in \cite{Guckenheimer2003} where the same rate-tipping phenomenon occurs. We examine both of these because they have a steadily drifting stable equilibrium near a fold, and the canard explosion in the van der Pol system is well-established. Weiczorek et. al. \cite{Wieczorek2011} also examine fast/slow systems with a stable equilibrium near the fold of a critical manifold. However, Weiczorek et. al. study behavior for a bi-asymptotic (logistic) forcing function, while our analysis in this paper is concerned with constant linear forcing.

%%%%%%%%%%%%%%%%%%%%%%%
\section{Fast/Slow System with Rate-Dependent Tipping}

\begin{figure}
\begin{center}
\includegraphics[width=.45\textwidth]{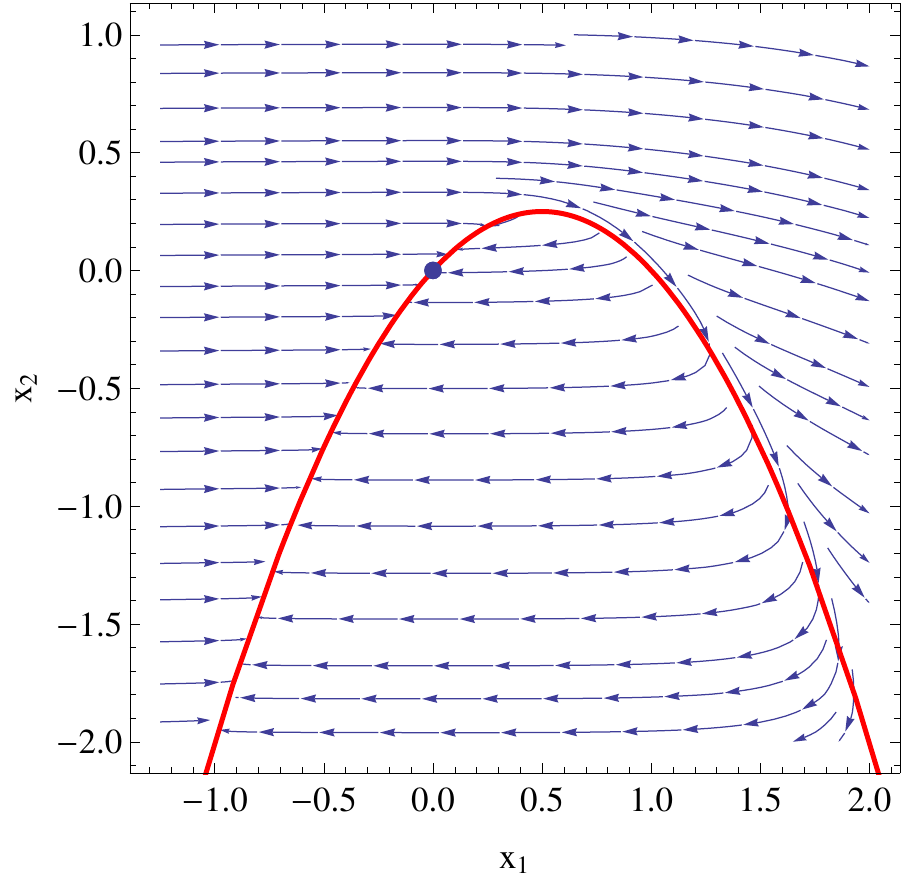}
%\captionsetup{format=plain}
\caption{Phase portrait of the fast-slow system in equation (\ref{fastSlowEquation1}) with a fixed $\lambda = 0$ and $\epsilon = .01$. The slow manifold is shown in red, and the dot shows the equilibrium point $(0, 0)$.}
\label{fig0}
\end{center}
\end{figure}

This section concerns a two-dimensional fast-slow system with a linear forcing function $\lambda(t) = rt$.
Ashwin et. al. analyze this system in the limit as the fast parameter $\epsilon \to 0$. Their analysis finds a critical tipping point at $r_c = \sum_{n=1}^{N}(1/2)^n$. Beyond the critical value, we find that the state begins to spiral around the quasi-static equilibrium.

The fast-slow system in \cite{Ashwin2012} (equations 3.12 - 3.14), with state variables $(x_1, x_2)$ and forcing parameter $\lambda$ increasing with rate $r$ is given by:

\begin{eqnarray}
\epsilon \frac{dx_1}{dt} &=& x_2 + \lambda + x_1(x_1 - 1) \label{fastSlowEquation1}\\
\frac{dx_2}{dt} &=& -\sum_{n=1}^{N} x_1^n\\
\frac{d\lambda}{dt} &=& r > 0
\end{eqnarray}

\noindent In this example $N$ is odd and $N \geq 5$, which ensures a globally stable equilibrium at $(0, -\lambda)$ with a fold at $(\frac{1}{2}, -\lambda + \frac{1}{4})$ for fixed $\lambda$. A phase portait of the system is shown in figure \ref{fig0}. Three numerically computed trajectories with different behaviors are depicted in figure \ref{fig1}. 

\subsection{Singular Perturbation Analysis}
We will go through a brief summary of the analysis in Ashwin et. al. \cite{Ashwin2012} to find the critical rate for $\epsilon \to 0$. Fixing $\lambda$, and setting $\epsilon = 0$ in the equation for $\frac{dx_1}{dt}$ gives the critical slow manifold as the set of points satisfying $0 = x_2 + \lambda + x_1(x_1 - 1)$, which is a folded manifold with a fold point at $(x_1, x_2) = (\frac{1}{2}, -\lambda + \frac{1}{4})$. This manifold has an attracting part where $x_1 < \frac{1}{2}$ and a repelling part where $x_1 > \frac{1}{2}$. The slow dynamics on the critical manifold can be approximated by differentiating this expression with respect to $t$:
\begin{eqnarray}
0 &= \frac{dx_2}{dt} + \frac{d\lambda}{dt} + \frac{dx_1}{dt}(2x_1 - 1)\label{reduced system 1}
\end{eqnarray}

%\begin{figure}
%\begin{center}
%\includegraphics[width=.4\textwidth]{figure2.eps}
%\captionsetup{format=plain}
%\caption{The critical manifold as a two-dimensional manifold in $(x_1, x_2, \lambda)$. The QSE $(0,-\lambda,\lambda)$ is plotted as a solid red line, and the fold of the critical manifold at $(1/2,-\lambda+1/4,\lambda)$ is a dashed red line.}
%\label{fig02}
%\end{center}
%\end{figure}

\begin{figure}
\begin{center}
\includegraphics[width=.44\textwidth]{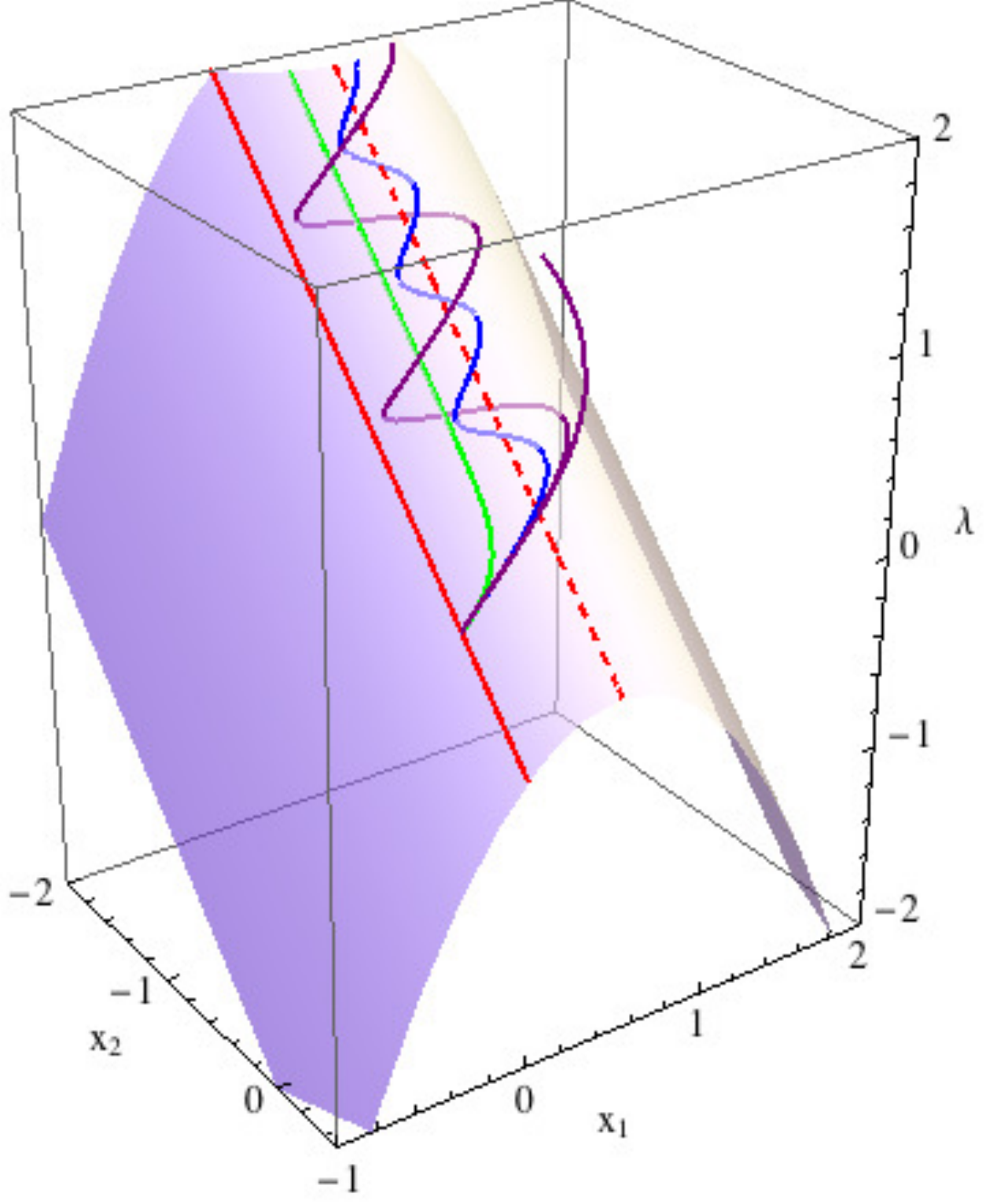}
\caption{The critical manifold as a two-dimensional manifold in $(x_1, x_2, \lambda)$. The QSE $(0,-\lambda,\lambda)$ is plotted as a solid red line, and the fold of the critical manifold at $(1/2,-\lambda+1/4,\lambda)$ is a dashed red line. The fast-slow system with trajectories shown for with various rates used. The parabolic critical manifold shown in the coordinates $(x_1, x_2, \lambda)$. The solid line shows the moving QSE, the dashed line shows the fold of the critical manifold. One trajectory has a rate $r < \sum_{n=1}^{N}(1/2)^n$, and follows the QSE at a steady distance. The other three trajectories have $r > \sum_{n=1}^{N}(1/2)^n$ and cycle around the fold with varying amplitudes. }
\label{fig1}
\end{center}
\end{figure}

\noindent Solving for $\frac{dx_1}{dt}$, the resulting dynamics of $x_1$ on the critical manifold is governed by the equation
\begin{eqnarray}
\frac{dx_1}{dt} &=  (\sum_{n=1}^{N} x_1^n - r)(2x_1 - 1)^{-1}\label{reduced system 2}
\end{eqnarray}

\noindent The system is singular for $x = \frac{1}{2}$, but rescaling time with $\frac{dt}{d\tau} = -(2x_1-1)$ we have the desingularized equation for $x_1$, with time reversed for $x > \frac{1}{2}$:
\begin{eqnarray}
\frac{dx_1}{d\tau} &=  r - \sum_{n=1}^{N} x_1^n\label{reduced system 3}
\end{eqnarray}

There is an equilibrium at $r = \sum_{n=1}^{N} x_1^n$, so for $r < \sum_{n=1}^{N} (1/2)^n$, all trajectories starting on the attracting part of the manifold ($x_1 < \frac{1}{2}$) will converge to $x_1^*$ satisfying $r = \sum_{n=1}^{N} (x_1^*)^n$. For $r > \sum_{n=1}^{N} (1/2)^n$, trajectories starting on the attracting part of the manifold will move toward the fold point at $x_1=\frac{1}{2}$ and slip over it, moving away from the QSE where $x_1=0$. 

Ashwin et. al. treat this as the tipping event with a critical rate at $r_c = \sum_{n=1}^{N} (1/2)^n$, where trajectories slip over the fold. But this doesn't quite represent the full story of this system. Numerical simulations show that for $r$ slightly greater than $r_c$, trajectories will stay close to the fold point, cycling around the fold, instead of simply diverging away from the fold and the QSE (see figure \ref{fig1}). It may be reasonable to consider this to be a form of tracking, with the state spiraling near the QSE. 

\subsection{Hopf Bifurcation in a Co-Moving System}

To explain why the state begins tracking in a spiral, we will first reduce the system to an autonomous 2-dimensional system with a change of variables. Ashwin et. al. call this a ``co-moving" system, as the new variable increases along with the variable $x_2$ and the rate parameter $\lambda$. In this system, we can find a Hopf bifurcation with an emerging periodic orbit that corresponds to the spiraling trajectories. We set $w = x_2 + \lambda$ and reduce the system to one that is autonomous:

\begin{eqnarray}
\epsilon \frac{dx_1}{dt} &= w + x_1(x_1 - 1) \label{comoving1}\\
\frac{dw}{dt} &= -\sum_{n=1}^{N} x_1^n + r \label{comoving2}
\end{eqnarray}

\begin{theorem}
The co-moving system $(\ref{comoving1}, \ref{comoving2})$ has a Hopf bifurcation at $r = \sum_{n=1}^{N}(1/2)^n$.
\end{theorem}
\begin{proof}
\noindent The equilibrium $(x_1^*, w^*)$ for (\ref{comoving1}, \ref{comoving2}) is given by the solution to 
\begin{eqnarray}
\sum_{n=1}^{N} x_1^{*n} &=& r\\
w^* &=& - x_1^{*2} + x_1^*
\end{eqnarray}
\noindent The Jacobian at this equilibrium is:
\begin{equation}\left(\begin{array}{lr}(2x_1^* - 1)/\epsilon & 1/{\epsilon}\\
\displaystyle{\sum_{n=1}^{N}  - n {(x^*_1)}^{n-1}} & 0
\end{array}\right)\end{equation}
and the eigenvalues of the Jacobian are
\begin{equation}
\displaystyle{\frac{2x^*_1 -1 \pm \sqrt{(1 - 2x^*_1)^2 - 4\epsilon\sum_{n=1}^N n {(x^*_1)}^{n-1}}}{2\epsilon} }.
\end{equation}

\begin{figure}
\begin{center}
\includegraphics[width=.45\textwidth]{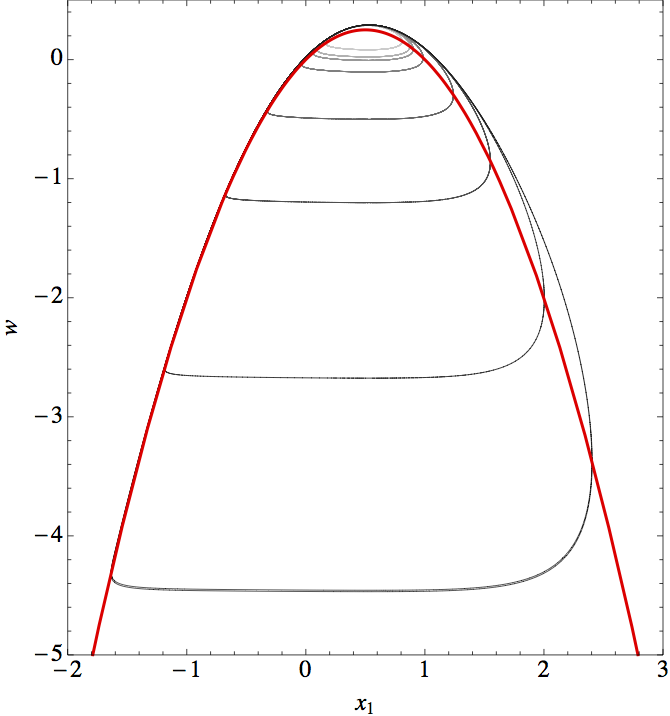} ~~~~~ \includegraphics[width=.45\textwidth]{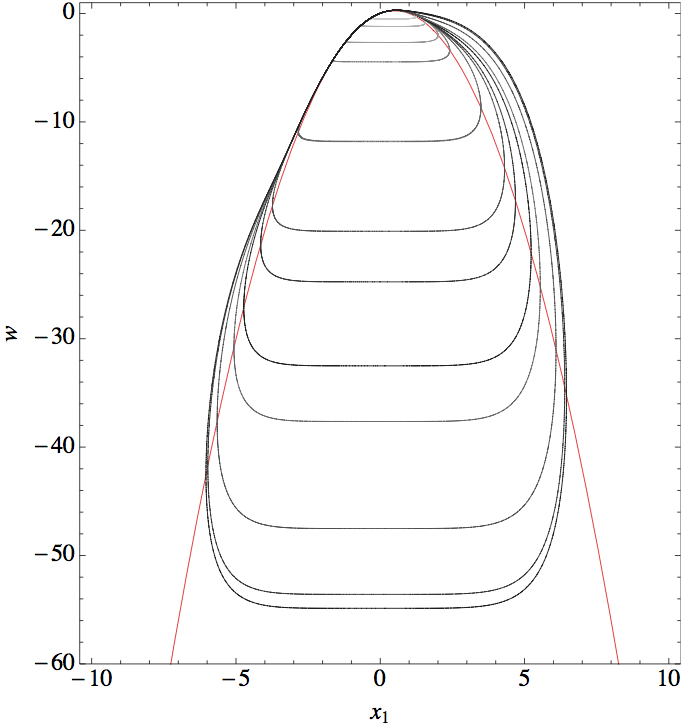}
\caption{The parabolic critical manifold (red) and several limit cycles in the co-moving system $(x_1, w)$ where $\epsilon = .02$ and $N = 5$, for several values of $r > \sum_{n=1}^{N}(1/2)^n = 0.9675$. The left shows the progression of limit cycles for $r \in [.9675, 1.092]$, and the right shows the continuous progression to large limit cycles as $r$ increases further, for $r \in [1.092, 1.3]$.}
\label{fig01}
\end{center}
\end{figure}

When $x_1^* < 1/2$, the equilibrium is asymptotically stable, which is in agreement with the conclusion in \cite{Ashwin2012} that the system does not tip for $r < \sum_{n=1}^{N}(1/2)^n = r_c$. As $r$ increases, so does $x_1^*$, so when $r = r_c$, and $x_1^* = 1/2$ the pair of eigenvalues are 

\begin{equation}
\displaystyle{\pm i\sqrt{\sum_{n=1}^N n (1/2)^{n-1}/\epsilon}}.
\end{equation}
Thus, the eigenvalues cross the imaginary axis and a Hopf bifurcation occurs. 
\end{proof}

\begin{corollary}
There is a positive neighborhood of the critical rate (i.e. $r \in (r_c, r_c + \delta)$ for some $\delta > 0$), for which trajectories follow spiraling curves $(x_1^*(t), w^*(t) - rt, rt)$, where $(x_1^*(t), w^*(t))$ is the asympototically stable limit cycle in the co-moving system $(\ref{comoving1}, \ref{comoving2})$.
\end{corollary} 

\begin{figure}
\begin{center}
\includegraphics[width=.65\textwidth]{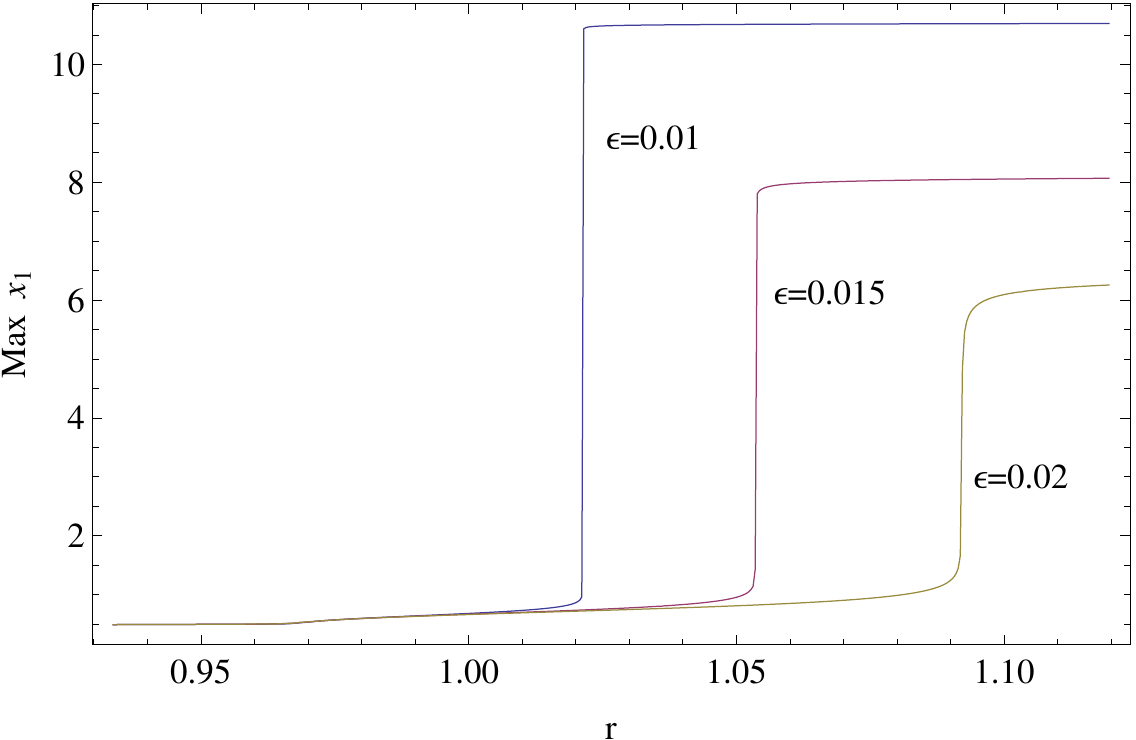}
\caption{The numerically determined maximum distance of the the limit cycle from the QSE as a function of the rate, $r$, plotted for $\epsilon = .015$ and $\epsilon = .02$. The rate for which there is a Hopf bifurcation is $r=0.96875$.}
\label{fig1a}
\end{center}
\end{figure}

As $r$ increases beyond $r_c$, the size of this limit cycle grows continuously. In theory it may be possible to prove analytically that there is a single attracting limit cycle. This problem could prove as challenging as a proof of the continuous progression of canard cycles in the van der Pol system. At the very least, however, it should be possible to find an isolating block with the vector field pointing inward on its boundary, implying at the very least the existence of an asymptotically stable invariant set of finite size. Rather than take on this tedious exercise, we show the progression of cycles numerically in figure 4.

These numerical simulations show the system has an attracting limit cycle, which grows quickly for $r > r_c$. With a periodic orbit growing continuously, it seems to be somewhat arbitrary how to answer when the system has tipped. The Hopf bifurcation proves there values of $r$ greater than $r_c$ such that the system cycles around the fold in a small periodic orbit. Values of $r$ below $r_c$ will not produce tipping, but values past that may produce a large excursion in the system, especially for very small $\epsilon$, or there only be relatively small periodic orbit around the fold. Ultimately, the value of $r_c$ certainly describles a bifurcation point, but the increase from tracking to a large excursion is not immediate at the value $r_c$. The maximum distance of the state to the QSE is a continuous function of $r$ even at $r_c$. We show this maximum distance for several values of $\epsilon$ in figure \ref{fig1a}. Not only is the maximum distance a continuous function of $r$, but it appears the greatest increase in maximum distance from the QSE occurs for a rate larger than the critical rate for trajectories surpassing the fold.

\begin{remark}
In the autonomous co-moving system, it is clear from Conley index theory of invariant sets for flows that changing the rate could not simply change a single global attractor to a repeller, without the creation of a new attracting invariant set. In the non-autonomous system with $\lambda = rt$, it is not immediately clear that this is true, but in this case it is. An index theory for isolated invariant sets in non-autonomous systems would be a useful in determining if or when such a non-autonomous bifurcation could occur.
\end{remark}

\section{Forced van der Pol Oscillator}

In this section, we will analyze a linearly forced van der Pol oscillator to demonstrate the progression of the spiraling trajectories in a well-known system. This system has been well studied; we are only transforming the equation slightly to demonstrate a rate-dependent bifurcation. The forced van der Pol oscillator as a 2nd order ODE is given by:

\[\frac{d^2x}{d\tau^2} -\mu(x^2-1)\frac{dx}{d\tau} + x = f(\tau) \]
We modify this form slightly, writing the equation as:
\[\frac{d^2x}{d\tau^2} -\mu(x^2-1)\frac{dx}{d\tau} + x = g(\tau) - \alpha\]
where $\alpha$ is a constant. Next, with the time transformation $t = -\tau/\mu$ we have
\[\frac{1}{\mu^2}\frac{d^2x}{dt^2} + (x^2-1)\frac{dx}{dt} + x = g(-\mu t) - \alpha\]
Setting $x_1 = x$ and using the Lienard transformation $x_2 = \frac{1}{\mu^2}\frac{dx_1}{dt} - (x_1 - \frac{x_1^3}{3}) + \int_0^t g(-\mu s) ds$, and setting $\lambda(t) = \int_0^t g(-\mu s) ds$ and $\epsilon = \frac{1}{\mu^2} << 1$, we end up with a fast/slow system of equations:

\begin{eqnarray}
\epsilon \frac{dx_1}{dt} &=& x_2 + (x_1 - \frac{x_1^3}{3}) + \lambda(t)\label{vanDerPol1}\\
\frac{dx_2}{dt} &=& -x_1 - \alpha
\end{eqnarray}

Van der Pol oscillators have been historically transformed so the forcing appears in the differential equation for the $x_2$ variable \cite{Guckenheimer2003, Szmolyan2000}. With this transformation the forcing $\lambda(t)$, is in the $x_1$ equation as in the previous fast/slow system. As before, we will assume $\lambda$ provides a linear forcing at a given rate $r$, so $\lambda(t) = rt$.

\begin{figure}
\begin{center}
\includegraphics[width=.48\textwidth]{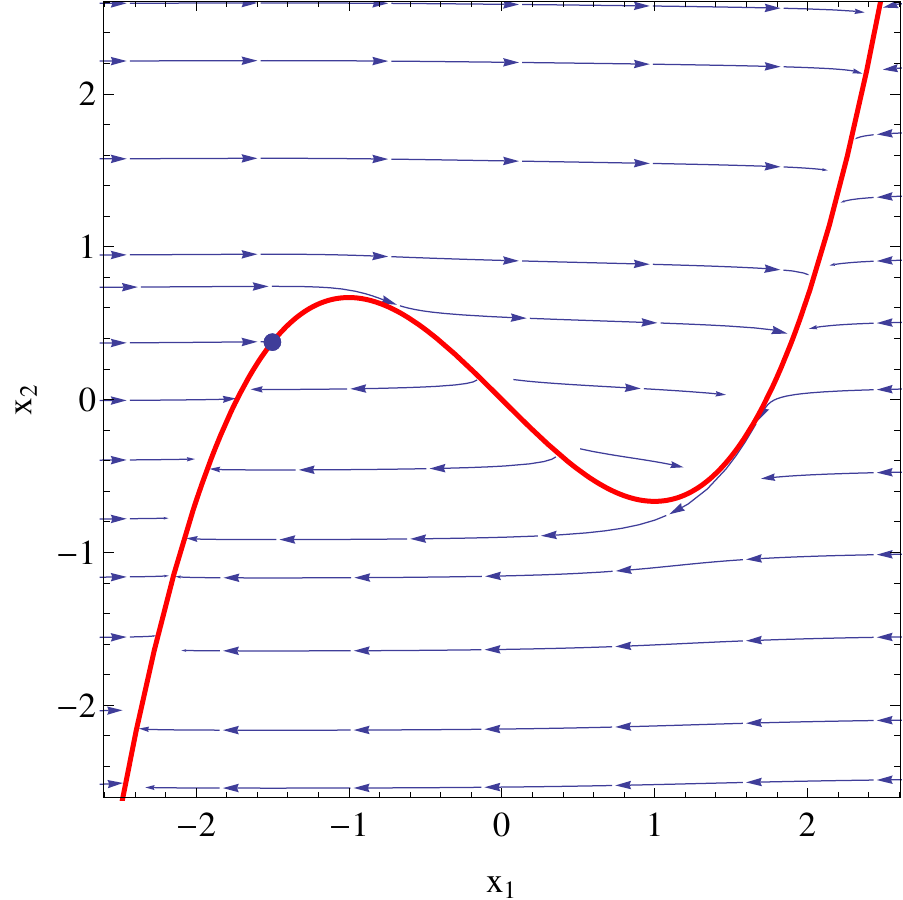} ~~~ \includegraphics[width=.48\textwidth]{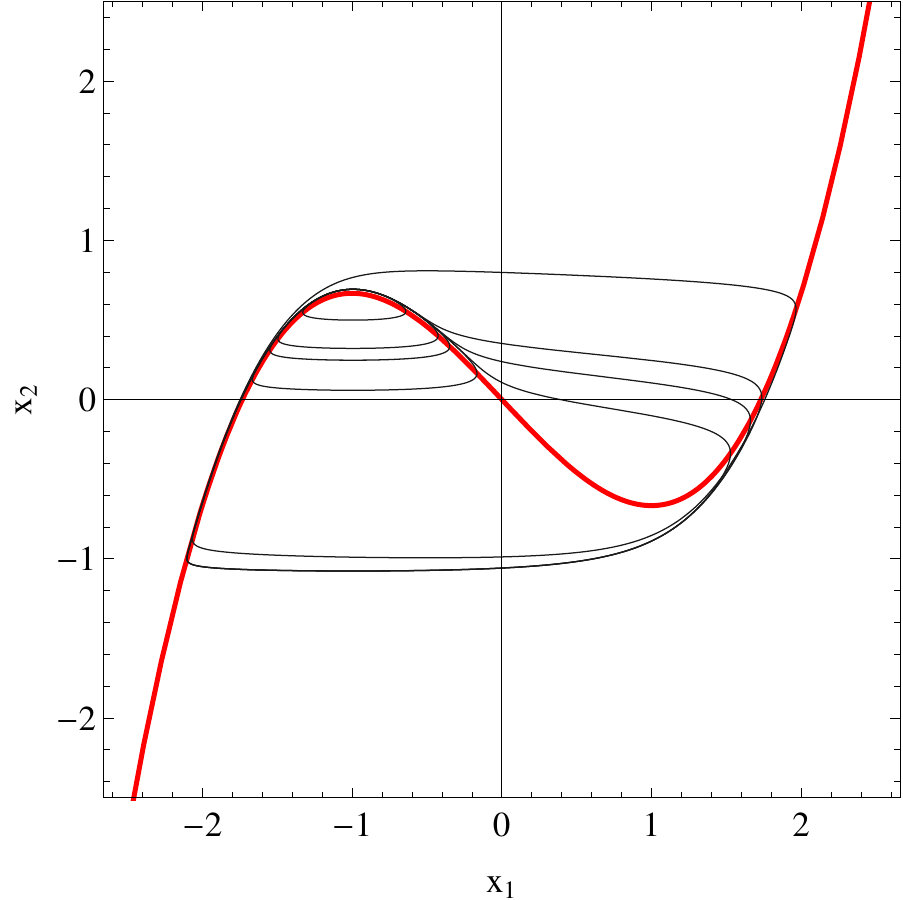}
\caption{Left: phase portrait of the system in (\ref{vanDerPol1}) for fixed $\lambda = 0$. The cubic slow manifold is in red, with the dot marking the stable equilibrium. Right: trajectories plotted illustrating the progression of the Hopf bifurcation and canard explosion in (\ref{vanDerPol-co-moving}).}
\label{fig2}
\end{center}
\end{figure}

In equation (\ref{vanDerPol1}), we will suppose that $\alpha > 1$ is a fixed constant, and for fixed $\lambda$, this system has a stable equilibrium at $(x_1^*, x_2^*) = (-\alpha, -\alpha + \frac{\alpha^3}{3} - \lambda)$. A phase portrait with $\lambda = 0$ is shown in figure \ref{fig2}. We create the co-moving system with the new variable, $w = x_2 + \lambda$:

\begin{eqnarray}
\epsilon \frac{dx_1}{dt} &=& w + (x_1 - \frac{x_1^3}{3})\label{vanDerPol-co-moving}\\
\frac{dw}{dt} &=& -x_1 - \alpha + r
\end{eqnarray}

\begin{figure}
\begin{center}
\includegraphics[width=.64\textwidth]{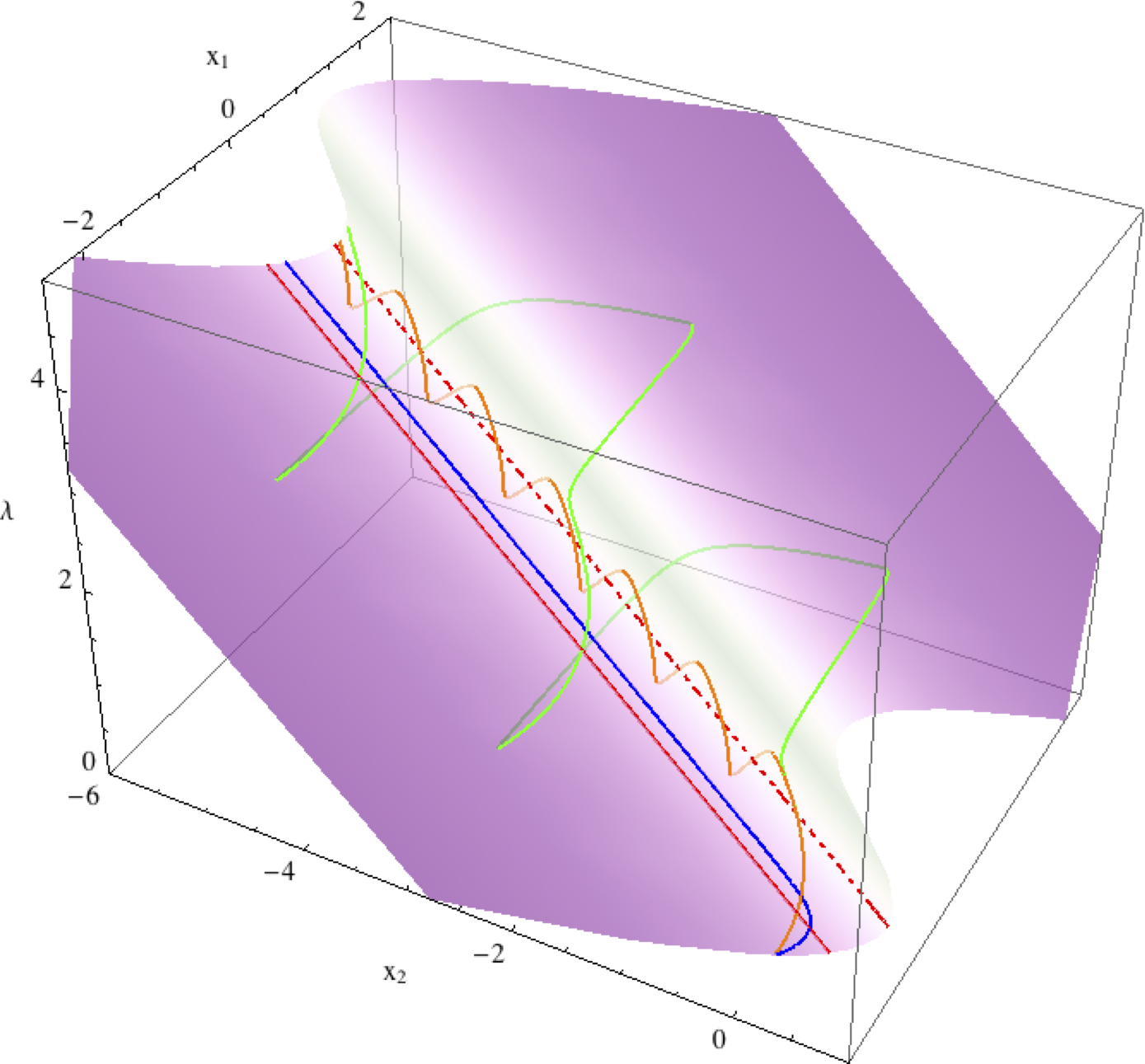}
\caption{The slow critical manifold of the forced van der Pol system is shown, overlaid with trajectories for various rates for the forcing parameter, with $\alpha = 1.5, \epsilon = .02$. Each trajectory has a different behavior: tracking the QSE closely ($r = 0.1)$, tracking in a small spiral (orange, $r = 0.5025$), and tracking as a large canard trajectory ($r = 0.503$).}
\label{fig4}
\end{center}
\end{figure}

This reduction results in the standard van der Pol system with $r$ as the bifurcation parameter \cite{Benoit1981, Diener1986}. In this system, a Hopf bifurcation will occur when $r = \alpha - 1$. The periodic orbit formed at the Hopf bifurcation will progress quickly through a canard explosion (see figure \ref{fig2}), a rate-dependent tipping point at $r = \alpha-1$. Like in the previous system, as the tipping occurs, the state will spiral around the equilibrium. If the rate is high enough in (\ref{vanDerPol-co-moving}), this spiral will be in the shape of the large canard trajectory. For even higher rates, at $r > \alpha + 1$, the co-moving system again has a stable equilibrium on the other branch of the critical manifold. The state will no longer spiral, but will follow at a much further distance from the QSE. 

\begin{figure}
\begin{center}
\includegraphics[width=.65\textwidth]{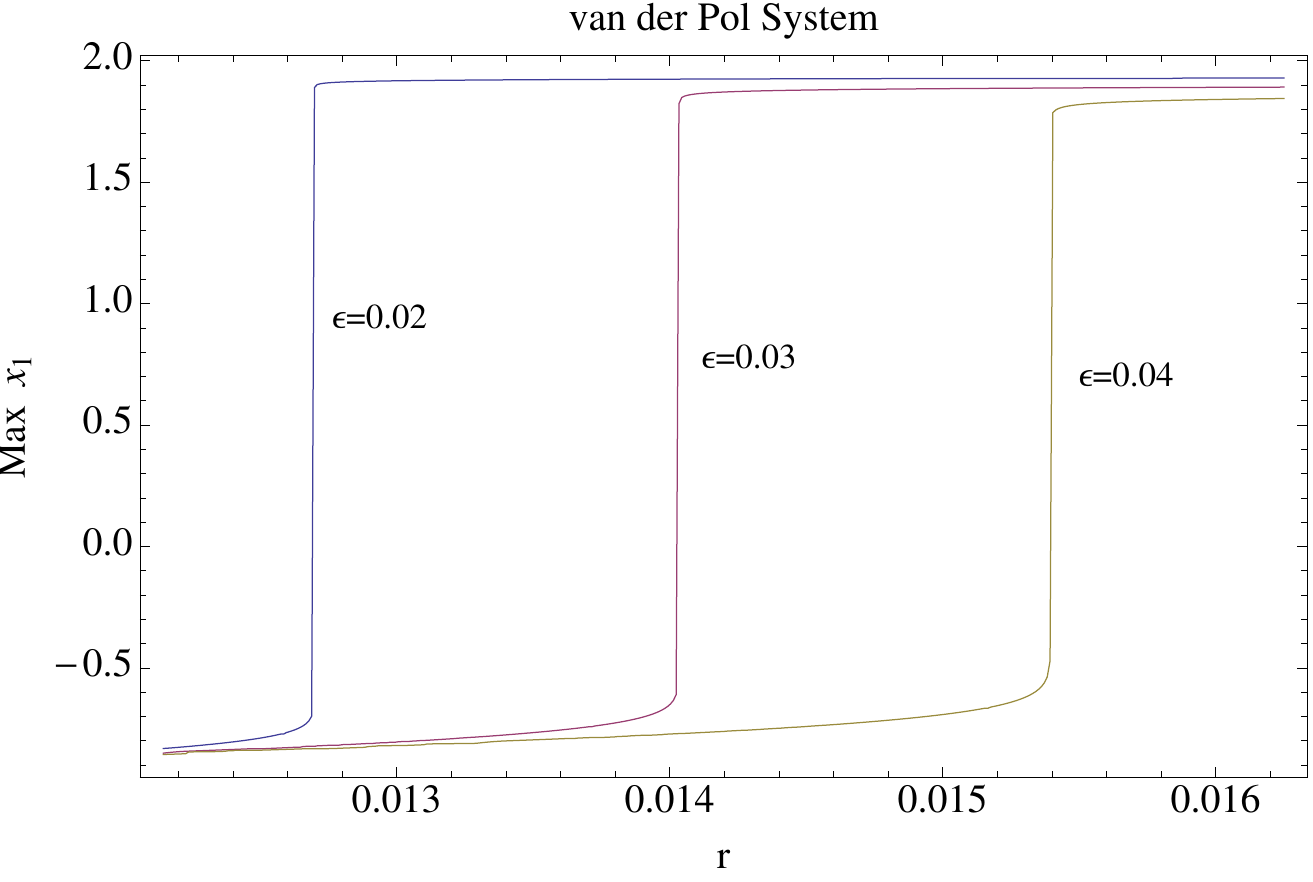}
\caption{The canard explosion in the van der Pol System. Plotted is the maximum distance of the the limit cycle as $r$ increases, plotted for $\epsilon = .02$, $\epsilon = .03$, and $\epsilon = .04$. For this figure we set $\alpha = 1.01$, so the Hopf bifurcation occurs at $r=0.01$. }
\label{fig9}
\end{center}
\end{figure}

The importance of this analysis is that the continuous progression of a single attracting limit cycle in the van der Pol system is well known, and we do not need to rely on numerical simulation to show the continuity with increasing $r$. Once again, at the point where trajectories slip over the fold, they continue to track the QSE in a spiral. We postulate that this behavior occurs in this form of tipping over a fold in general.

\section{Conclusions}

We presented cases of rate-dependent tipping with a quasi-stable equilibrium near the fold of a critical manifold in a fast/slow dynamical system. Each system is forced at a steady rate, and at a critical rate, GSP analysis shows that trajectories starting on the stable part of the critical manifold will pass over the fold, a tipping event. However, our analysis of the co-moving system shows that for $\epsilon > 0$, trajectories do not simply diverge away from the QSE. Rather, they start to track in a spiral. The example we presented was simple to reduce and analyze, but the behavior is likely to appear in more generic systems with equilibrium points near the fold of a critical manifold.

A key point of this paper is that at the critical rate value, these systems do not just proceed immediately from following the QSE to repelling far away from the QSE. The maximum distance of the limiting spiraling trajectory is a continuous function of the rate. 

In the autonomous system, if all trajectories were to diverge from the QSE at a critical rate, it would imply the disappearance of an attracting invariant set from an isolating block, an event which can be ruled out with basic index theory of invariant sets. The non-autonomous nature of rate-induced bifurcations makes us desire an index theory for attractors in non-autonomous systems, in cases when we cannot reduce the system to an autonomous one. An index theory for invariant sets at singular limits (as $\epsilon \to 0$) may be worth pursuing as well.

This type of rate-induced tipping is likely to show up in other fast/slow systems with folded slow manifolds. We examined linear forcing in these systems, but we suspect there is more to be discovered about the bifurcation structure in examples of rate-induced tipping with various kinds of forcing functions. This result may also have implications for the prediction of rate-induced tipping points in real-world systems.

\vspace{.5in}

\bibliography{FastSlowBibliography}

\begin{thebibliography}{1}

\bibitem{Ashwin2015}
P.~Ashwin, C.~Perryman, and S.~Wieczorek.
\newblock Parameter shifts for nonautonomous systems in low dimension:
  bifurcation and rate-induced tipping.
\newblock {\em arXiv:1506.07734}, 2015.

\bibitem{Ashwin2012}
P.~Ashwin, S.~Wieczorek, R.~Vitolo, and P.~Cox.
\newblock Tipping points in open systems: bifurction, noise-induced, and
  rate-dependent examples in the climate system.
\newblock {\em Phil. Trans. R. Soc.}, 370:1166 -- 1184, 2012.

\bibitem{Benoit1981}
E.~Benoit, J.~L. Caixot, F.~Diener, and M.~Diener.
\newblock Chasse de canard (premi{\`e}re partie).
\newblock {\em Collectanea Mathematica}, 32(1):37--76, 1981.

\bibitem{Diener1986}
M.~Diener.
\newblock The canard unchained or how fast/slow dynamical systems bifurcate.
\newblock {\em The Mathematical Intelligencer}, 6(3):38--49, 1984.

\bibitem{Guckenheimer2003}
J.~Guckenheimer, K.~Hoffman, and W.~Weckesser.
\newblock The forced van der {P}ol equation {I}: the slow flow and its
  bifurcations.
\newblock {\em SIAM J. Applied Dynamical Systems}, 2:1--35, 2003.

\bibitem{Perryman2015}
C.~Perryman.
\newblock How fast is too fast? {R}ate-induced bifurcations in multiple
  time-scale systems.
\newblock {\em PhD Thesis: University of Exeter}, 2015.

\bibitem{Scheffer2008}
M.~Scheffer, E.~H. van Nes, M.~Holmgren, and T.~Hughes.
\newblock Pulse-driven loss of top down control: the critical-rate hypothesis.
\newblock {\em Ecosystems}, 11:226--237, 2008.

\bibitem{Szmolyan2000}
P.~Szmolyan and M.~Weschelberger.
\newblock Canards in {$\mathbb{R}^3$}.
\newblock {\em Journal of Differential Equations}, 177:419–453, 2001.

\bibitem{Wieczorek2011}
S.~Wieczorek, P.~Ashwin, C.~M. Luke, and P.~M. Cox.
\newblock Excitability in ramped systems: the compost bomb instability.
\newblock {\em Proc. R. Soc. A}, 467:1243--1269, 2011.

\end{thebibliography}
\bibliographystyle{plain}

\end{document}